\documentclass[12pt]{amsart}
\usepackage{amsmath,amssymb,epsfig,color,url}
\newtheorem{theorem}{Theorem}[section]

\newtheorem{corollary}[theorem]{Corollary}
\newtheorem{definition}[theorem]{Definition}
\newtheorem{lemma}[theorem]{Lemma}
\newtheorem{proposition}[theorem]{Proposition}

\theoremstyle{remark}
\newtheorem{remark}[theorem]{Remark}

\newcommand{\vanish}[1]{}\parskip=12pt

\newcommand{\sP}{\widetilde{P}}
\newcommand{\ssP}{R}
\newcommand{\adn}{\widetilde{d}}

\begin{document}
\title{Shifted Jacobi polynomials and Delannoy numbers} 
\author{G\'abor Hetyei}
\address{Department of Mathematics and Statistics, UNC Charlotte, 
	Charlotte, NC 28223}
\email{ghetyei@uncc.edu}
\dedicatory{\`A la m\'emoire de Pierre Leroux}
\thanks{This work was supported by the NSA grant
\# H98230-07-1-0073.}
\subjclass{Primary 05E35; Secondary 33C45, 05A15}
\keywords{Legendre polynomials, Delannoy numbers, Jacobi polynomials, rook
polynomials, Laguerre polynomials, Romanovski polynomials,
Narayana polynomials}  
\begin{abstract}
We express a weighted generalization of the Delannoy numbers in terms of
shifted Jacobi polynomials. A specialization of our formulas extends  
a relation between the central Delannoy numbers and 
Legendre polynomials, observed over 50 years
ago, to all Delannoy numbers and certain Jacobi
polynomials. Another specialization provides a weighted lattice path
enumeration model for shifted Jacobi polynomials, we use this to 
present a new combinatorial proof of the orthogonality
of Jacobi polynomials with natural number parameters. The proof relates the
orthogonality of these polynomials to the orthogonality of (generalized)
Laguerre polynomials, as they arise in the theory of rook
polynomials. We provide a combinatorial proof for the orthogonality 
of certain Romanovski-Jacobi polynomials with zero first parameter and
negative integer second parameter, considered as an initial segment 
of the list of similarly transformed Jacobi polynomials with the same
parameters. 
We observe that for an odd second parameter one more polynomial may be added to
this finite orthogonal polynomial sequence than what was predicted by a
classical result of Askey and Romanovski. The remaining transformed Jacobi
polynomials in the sequence are either equal to the already listed
Romanovski-Jacobi polynomials or monomial multiples of similarly
transformed Jacobi polynomials with a positive second parameter.    
We provide expressions for an analogous weighted generalization of the
Schr\"oder numbers in terms of the Jacobi polynomials, and use 
this model, together with a result of Mansour and Sun, to express 
the Narayana polynomials in terms of shifted Jacobi polynomials.
\end{abstract}
\maketitle

\section*{Introduction}
It has been noted more than fifty years ago~\cite{Good,Lawden,Moser},
that the diagonal entries of the Delannoy 
array $(d_{m,n})$, introduced by Henri Delannoy~\cite{Delannoy}, and the
Legendre polynomials $P_{n}(x)$ satisfy the equality 
\begin{equation}
\label{E_dp}
d_{n,n}=P_n(3),
\end{equation}
but this relation was mostly considered a ``coincidence''. The 
main result of our present work is that a weighted lattice path
enumeration model may be used to show that (\ref{E_dp}) can 
be extended to 
\begin{equation}
\label{E_dp1}
d_{n+\alpha,n}=P^{(\alpha,0)}_n(3) \quad\mbox{for all $\alpha\in {\mathbb
    Z}$ such that $\alpha\geq -n$,}
\end{equation}
where $P^{(\alpha,0)}_n(x)$ is the Jacobi polynomial with parameters
$(\alpha,0)$. This result indicates that
the interaction between Jacobi polynomials (generalizing Legendre
polynomials) and the Delannoy numbers is more than a mere
coincidence. We then use this model to
provide a new combinatorial proof of the orthogonality of the
Jacobi polynomials with nonnegative integer parameters. The first
combinatorial proof of this orthogonality is due to Foata and
Zeilberger~\cite{Foata-Zeilberger}, who provided a combinatorial model
for the weighted integral of an arbitrary product of Jacobi polynomials
with arbitrary parameters $\alpha, \beta$. The weighted lattice path
enumeration model presented here is used to prove orthogonality only,
and we need to assume $\alpha,\beta\in {\mathbb N}$. In exchange,  
our model is ``more combinatorial'' in the sense that the parameters
$\alpha,\beta\in {\mathbb N}$ also ``count something'' and are not just
``mere weights''. Using our model we may also consider Jacobi
polynomials with negative integer parameters and show that, 
for a positive integer $\beta$, the first $\beta-1$ entries in the sequence
of Jacobi polynomials $\{P^{(0,-\beta)}_n(x)\}_{n\geq 0}$ form a
symmetric sequence, while the entries of degree at least $\beta$ are monomial
multiples of Jacobi polynomials of the form $P^{(0,\beta)}_{n-\beta}(x)$. 
The first $\lfloor (\beta-1)/2\rfloor$ entries in this sequence form a 
finite sequence of orthogonal polynomials, discovered by
Romanovski~\cite{Romanovski}, although he and Askey~\cite{Askey} made
the statement in greater generality but about the first $\lfloor
(\beta-2)/2\rfloor$ entries only.    
We observe that in this special case, for odd $\beta$ the sequence may
be extended by an additional entry, and provide a combinatorial proof
of this amended statement. Finally, by extending our model to Schr\"oder paths,
we show that ordinary Schr\"oder
numbers too may be obtained by substituting $3$ into an appropriate
Jacobi polynomial, and prove a formula for iterated antiderivatives of
the shifted Legendre polynomials. This formula, together with a result
of Mansour and Sun~\cite{Mansour-Sun}, allows us to express the Narayana
polynomials in terms of shifted Jacobi polynomials. 

Our paper is structured as follows. In Section~\ref{s_wdn} we introduce
weighted Delannoy paths associating weight $u$ to each east
step, weight $v$ to each north step and weight $w$ to
each northeast step. This model was first considered by Fray and
Roselle~\cite{Fray-Roselle}. Using their formulas, we show that the
total weight of all Delannoy paths in a rectangle may be
expressed by substitution into a shifted Jacobi polynomial with the
appropriate parameters. The shifting is achieved by replacing $x$ with
$2x-1$ in the definition of the Jacobi polynomial.  

A specific substitution into the parameters $u$, $v$, $w$
yields the shifted Jacobi polynomials themselves. Using this observation, in
Section~\ref{s_orth} we present our lattice path enumeration proof for the
orthogonality of the Jacobi polynomials $P^{(\alpha,\beta)}_n(x)$ when
$\alpha,\beta\in {\mathbb N}$. The proof reduces the
weight enumeration to a known formula in the classical theory of
rook polynomials, used to show the orthogonality of the
(generalized) Laguerre polynomials. Further discussion of this
connection may be found in Section~\ref{s_rook}.

The discussion of the properties of the sequence of Jacobi polynomials
$\{P^{(0,-\beta)}_n(x)\}_{n\geq 0}$ is in
Section~\ref{s_bnneg}, while the results associated to our weighted 
generalization of Schr\"oder numbers may be found in in Section~\ref{s_S}. 
The paper inspires as many questions as it answers, some of which are
listed in the concluding Section~\ref{sec:concl}.

\section{Preliminaries}
\subsection{Delannoy numbers}

The {\em Delannoy array} $(d_{i,j}: i,j\in {\mathbb Z})$
was introduced by Henri Delannoy~\cite{Delannoy}.
It may be defined by the recursion formula 
\begin{equation}
\label{E_Drec}
d_{i,j} = d_{i-1,j} + d_{i,j-1} + d_{i-1,j-1}
\end{equation}
with  the conditions $d_{0,0} = 1$ and $ d_{i,j} = 0 $ if $i<0$ or $j<0$.
The historic significance of these numbers is explained in the paper
``Why Delannoy numbers?''\cite{Banderier-Schwer} by 
Banderier and Schwer. The diagonal elements $(d_{n,n}: n\geq 0)$ in
this array are the {\em central Delannoy numbers} (A001850 of Sloane
\cite{Sloane}). These numbers are known through the books of
Comtet~\cite{Comtet} and Stanley~\cite{Stanley-EC2}, but it is Sulanke's
paper~\cite{Sulanke} that gives the most complete enumeration of all known
uses of the central Delannoy numbers. 

\subsection{Jacobi, Legendre and Romanovski polynomials}

The $n$-th Jacobi polynomial $P_n^{(\alpha,\beta)}(x)$  of type
$(\alpha,\beta)$ is defined as 
$$
P_n^{(\alpha,\beta)}(x)=(-2)^{-n}(n!)^{-1} (1-x)^{-\alpha}(1+x)^{-\beta}
\frac{d^n}{dx^n}\left((1-x)^{n+\alpha}(1+x)^{n+\beta}\right).
$$
If $\alpha$ and $\beta$ are real numbers satisfying $\alpha,\beta>-1$,
then the polynomials $\{P_n^{(\alpha,\beta)}(x)\}_{n\geq 0}$  form an
orthogonal basis with respect to the inner product 
$$\langle f, g\rangle :=\int_{-1}^1 f(x)\cdot g(x)\cdot
(1-x)^{\alpha}(1+x)^{\beta}\ dx.$$ 
The following formula, stated only slightly differently in
\cite[(4.21.2)]{Szego}, may be used to extend the definition of
$P_n^{(\alpha,\beta)}(x)$ to arbitrary complex values of $\alpha$ and
$\beta$:  
\begin{equation}
\label{E_J1}
P_n^{(\alpha,\beta)}(x)
=\sum_{j=0}^n\binom{n+\alpha+\beta+j}{j}\binom{n+\alpha}{n-j} 
\left(\frac{x-1}{2}\right)^j.
\end{equation}
The following relation provides a way to ``swap'' the parameters
$\alpha$ and $\beta$ (see~\cite[Ch.\ V, (2.8)]{Chihara}).
\begin{equation}
\label{E_jswap}
(-1)^n P_n^{(\alpha,\beta)}(-x)=P_n^{(\beta,\alpha)}(x).
\end{equation}

The Legendre polynomials $\{P_n(x)\}_{n\geq 0}$ are the Jacobi polynomials
$\{P_n^{(0,0)}(x)\}_{n\geq 0}$. They form an orthogonal  basis with
respect to the inner product 
$$\langle f, g\rangle :=\int_{-1}^1 f(x)\cdot g(x)\ dx.$$ 
Substituting $\alpha=\beta=0$ into (\ref{E_J1}) yields
\begin{equation}
\label{E_P1}
P_n(x)
=\sum_{j=0}^n\binom{n+j}{j}\binom{n}{n-j} 
\left(\frac{x-1}{2}\right)^j.
\end{equation}
As a specialization of (\ref{E_jswap}) we obtain
\begin{equation}
\label{E_pswap}
(-1)^n P_n(-x)=P_n(x).
\end{equation}
The shifted Legendre polynomials $\sP_n(x)$ are defined by the linear 
substitution
$$
\sP_n(x):=P_n(2x-1).
$$
They form an orthogonal basis with respect to the inner product 
$$\langle f, g\rangle :=\int_{0}^1 f(x)\cdot g(x)\ dx.$$ 
They may be calculated using the following formula:
\begin{equation}
\label{E_sP}
\sP_n(x)
=\sum_{k=0}^n (-1)^{n-k} \binom{n}{k}\binom{n+k}{k} x^k
=\sum_{k=0}^n (-1)^{n-k}\binom{n+k}{n-k}\binom{2k}{k} x^k.
\end{equation}
A generalization of (\ref{E_sP}) is stated and shown in
Proposition~\ref{P_sj}. Shifted Legendre polynomials are widely used,
they even have a table entry in the venerable opus of Abramowitz
and Stegun~\cite[22.2.11]{Abramowitz-Stegun}. Their obvious
generalization, the {\em shifted Jacobi
  polynomials $\sP^{(\alpha,\beta)}_n(x)$}, defined by the formula
\begin{equation}
\label{P_sjdef}
\sP^{(\alpha,\beta)}_n(x):=P^{(\alpha,\beta)}_n(2x-1),
\end{equation} 
seem to receive rare mention by this name, a sample reference
is the work of Gatteschi~\cite{Gatteschi}. On the other hand, a finite
sequence of orthogonal polynomials, discovered by
Romanovski~\cite[p.\ 1025]{Romanovski} is essentially the same as a
slightly differently shifted sequence of Jacobi
polynomials. Romanovski's result  
was rephrased by Askey~\cite[(1.15) and (1.16)]{Askey} as follows.
\begin{theorem}[Askey-Romanovski]
\label{thm:Askey-Romanovski}
When $\alpha>-1$ and $\alpha+\beta+N+1<0$ then for $0\leq m,n\leq N/2$
the polynomials 
$$
\ssP_n^{(\alpha,\beta)}(x)=\frac{(\alpha+1)_n}{n!}
{}_2F_1(-n,n+\alpha+\beta+1; \alpha+1; x)
$$
satisfy the orthogonality relation 
$$
\int_0^{\infty} \ssP_n^{(\alpha,\beta)}(x) R_m^{(\alpha,\beta)}(x) x^{\alpha}(1+x)^{\beta}\ dx =
\frac{(-1)^{n+1} \Gamma(-n-\alpha-\beta) \Gamma(n+\alpha+1) (\beta+1)_n
\cdot \delta_{m,n}}{(2n+\alpha+\beta+1) n! \Gamma(-\beta)}.
$$

\end{theorem}
Here and in the rest of this paper, $\delta_{m,n}$ is the Kronecker
delta, $(\alpha+1)_n=(\alpha+1)(\alpha+2)\cdots (\alpha+n)$ is the {\em
  Pochhammer symbol} and $\Gamma(z)$ is the gamma
function~\cite[6.1.1]{Abramowitz-Stegun}. The the Gaussian hypergeometric
function ${}_2F_1$ will not used beyond the end of this subsection. 

Comparing the above definition of $\ssP_n^{(\alpha,\beta)}(x)$ with the
analogous formula for the Jacobi polynomials found in the database
created by Koekoek and Swarttouw~\cite[(1.8.1)]{Koekoek-Swarttouw} we obtain 
\begin{equation}
\label{eq:ssPP}
\ssP_n^{(\alpha,\beta)}(x)=P_n^{(\alpha,\beta)}(2x+1).
\end{equation}
This relation was noted by Chen and Srivastava~\cite[(4.7)]{Chen-Srivastava}.
\begin{remark}
In his very concise note, Romanovski~\cite{Romanovski} considered
several finite sequences of orthogonal polynomials. Not all finite
orthogonal polynomial sequences called ``Romanovski polynomials'' in the
literature coincide with the polynomials $\ssP_n^{(\alpha,\beta)}(x)$
above. To avoid confusion, Lesky~\cite{Lesky} distinguishes between 
Romanovski-Bessel, Romanovski-Jacobi and Romanovski-Pseudo Jacobi
polynomials. The polynomials $\ssP_n^{(\alpha,\beta)}(x)$ above are
examples of Romanovski-Jacobi polynomials. 
\end{remark}

\subsection{Favard's theorem}

Our main reference for the basic facts on orthogonal polynomials is
Chihara's book~\cite{Chihara}. A {\em moment functional} ${\mathcal L}$
is a linear map ${\mathbb C}[x]\rightarrow {\mathbb C}$. A sequence of
polynomials $\{p_n (x)\}_{n=0}^{\infty}$ is an {\em orthogonal polynomial
sequence} with respect to ${\mathcal L}$ if $p_n(x)$ has degree $n$, 
${\mathcal L}(p_m(x)p_n(x))=0$ for $m\neq n$, and ${\mathcal
  L}(p_n^2 (x))\neq 0$ for all $n$. Such a sequence exists if and only if
${\mathcal L}$ is {\em quasi-definite} (see~\cite[Ch. I, Theorem
  3.1]{Chihara},  the term quasi-definite is introduced in~\cite[Ch. I,
  Definition 3.2]{Chihara}). Whenever an orthogonal polynomial sequence
exists, each of its elements is determined up to a non-zero constant
factor (see \cite[Ch. I, Corollary of Theorem 2.2]{Chihara}).  

A way to verify whether a sequence of polynomials is orthogonal is
Favard's theorem~\cite[Ch.\ I, Theorem 4.4]{Chihara}. This states
that a sequence of monic polynomials $\{p_n(x)\}_{n\geq 0}$ is an
orthogonal polynomial sequence,
if and only if it satisfies the recurrence formula  
\begin{equation}
\label{E_Favrec}
p_n(x)=(x-c_n)p_{n-1}(x)-\lambda_n p_{n-2}(x) \quad\mbox{$n=1,2,3,\ldots$}
\end{equation}
where $p_{-1}(x)=0$, $p_0(x)=1$, the numbers $c_n$ and
$\lambda_n$ are constants, $\lambda_n\neq 0$ for $n\geq 2$, and
$\lambda_1$ is arbitrary (see \cite[Ch. I, Theorem 4.1]{Chihara}).
Conversely, for every
sequence of monic polynomials defined in the above way there is a unique
quasi-definite moment functional ${\mathcal L}$ such that ${\mathcal
  L}(1)=\lambda_1$ and $\{P_n (x)\}_{n=0}^{\infty}$  is the monic orthogonal polynomial sequence
with respect to ${\mathcal L}$. The original proof provides only a
recursive description of ${\mathcal L}$ for a given $\{p_n(x)\}_{n\geq
  0}$ satisfying (\ref{E_Favrec}). Viennot~\cite{Viennot} 
gave a combinatorial proof of Favard's theorem, upon
which he has built a general combinatorial theory of orthogonal
polynomials. In his theory, the values ${\mathcal L}(x^n)$ are explicitly
given as sums of weighted {\em Motzkin paths}. 

\subsection{Central Delannoy numbers and Legendre polynomials}
\label{s_Legendre}

Equation (\ref{E_dp}) linking the central Delannoy numbers to the 
Legendre polynomials has been known for over 50 years~\cite{Good},
\cite{Lawden}, \cite{Moser}. Until recently there
was a consensus that this link is not very relevant.
Banderier and Schwer~\cite{Banderier-Schwer} note
that there is no ``natural'' correspondence between Legendre polynomials
and the original lattice path enumeration problem associated to the
Delannoy array, while Sulanke~\cite{Sulanke} states that
``the definition of Legendre polynomials does not appear to foster any
combinatorial interpretation leading to enumeration''. 

The present author made two attempts to provide a combinatorial
explanation. In~\cite{Hetyei-dn} we find a generalization (\ref{E_dp}) 
to substitutions of $3$ into Jacobi polynomials and an asymmetric
variant of the Delannoy array, having the same diagonal
elements. In~\cite{Hetyei-lp} we find a construction of a polytope whose
face numbers are the coefficients of the powers of $(x-1)/2$ in (\ref{E_P1})
and the Delannoy numbers enumerate faces having nonempty intersections
with certain generalized orthants. Both attempts are ``imperfect'' in
some sense: the first does not relate the original Delannoy array to
substitutions into Jacobi polynomials, the second does not involve
lattice path enumeration. The present work overcomes both of these
``imperfections''. 

\subsection{Rook polynomials and orthogonal polynomials}

An excellent short summary of the classical theory of rook polynomials
is given by Gessel in~\cite[Section 2]{Gessel}. For more information we
follow his suggestion and refer the reader to Kaplansky and
Riordan~\cite{Kaplansky-Riordan} and
Riordan~\cite[pp. 163--277]{Riordan}. Let $B$ be a subset of
$[n]\times [n]$, i.e., a {\em board}. Here $[n]$
is a shorthand for $\{1,\ldots,n\}$. A subset $S$ of
$B$ is called {\em compatible} if no two elements of $S$ agree in either
coordinate. The {\em rook polynomial of $B$} is defined as 
$$
r_B(x):=\sum_{k=0}^n (-1)^k r_k x^{n-k}
$$
where $r_k$ is the number of compatible $k$-subsets of $B$. Let ${\mathcal L}$
be the linear functional on polynomials in $x$ defined by
${\mathcal L}(x^n):=n!$. Then 
$$
{\mathcal L}(p(x))=\int_0^{\infty} e^{-x} p(x) \ dx
$$
and the number of permutations $\pi$ of $[n]\times [n]$
such that no $(i,\pi(i))$ belongs to $B$ is ${\mathcal L}(r_B(x))$. 

The rook polynomial of $[n]\times [n]$ is the
{\em Laguerre polynomial}
\begin{equation}
\label{E_Ldef}
l_n(x):=\sum_{k=0}^n (-1)^k \binom{n}{k}^2 k! x^{n-k}
\end{equation}
In terms of the simple Laguerre polynomial as usually normalized,
$$
l_n(x)=(-1)^n n! L_n(x).
$$
It can be shown combinatorially that the Laguerre polynomials form an
orthogonal basis with respect to the inner product induced by ${\mathcal L}$. We have
$$
{\mathcal L}(l_m(x)l_n(x))=\delta_{m,n} n!
$$
Here $\delta_{m,n}$ is the Kronecker delta. 

\section{Weighted Delannoy numbers and shifted Legendre polynomials} 
\label{s_wdn} 

A {\em Delannoy path} is a lattice path using only three
kinds of steps: $(0,1)$ (east), $(1,0)$ (north), and $(1,1)$
(northeast). One of the most plausible generalizations of the Delannoy
array is the following. 
\begin{definition}
\label{D_wdn}
Let $u$, $v$, $w$ be commuting variables. We define the {\em weighted
  Delannoy numbers $d_{m,n}^{u,v,w}$} as the total weight of all
  Delannoy paths from $(0,0)$ to $(m,n)$, where each east step $(0,1)$
  has weight $u$, each north step has weight $v$, and each northeast
  step has weight $w$. The weight of a
  lattice path is the product of the weights of its steps. 
\end{definition}
The usual Delannoy numbers $d_{m,n}$ are obtained by substituting 
$u=v=w=1$. Without any reference to Delannoy numbers, a
short study of the numbers $d_{m,n}^{u,v,w}$ may be found in the
work of Fray and Roselle~\cite[Section 2]{Fray-Roselle}. They use the notation
$f(n,k):=d_{n,k}^{x,y,z}$. The recursion formula (\ref{E_Drec}) of the
ordinary Delannoy numbers generalizes to 
\begin{equation}
\label{E_wDrec}
d_{i,j}^{u,v,w} = u\cdot d_{i-1,j}^{u,v,w} + v\cdot d_{i,j-1}^{u,v,w} +
w\cdot d_{i-1,j-1}^{u,v,w},
\end{equation}
see~\cite[Eq.\ (2.1)]{Fray-Roselle}. 
The weighted Delannoy numbers also satisfy the following explicit
formula~\cite[Eq.\ (2.2)]{Fray-Roselle}, which was shown by Fray and
Roselle using generating functions:
\begin{equation}
\label{E_FR}
d_{m,n}^{u,v,w}=\sum_{k=0}^m \binom{m}{k}\binom{m+n-k}{m} u^{m-k} v^{n-k} w^k.
\end{equation}
Using the identity
$$
\binom{m}{k}\binom{m+n-k}{m}=\binom{m+n-k}{k,m-k,n-k}=
\binom{m+n-k}{k}\binom{m+n-2k}{n-k},
$$
Equation (\ref{E_FR}) may be restated as follows. 
\begin{proposition}[Fray-Roselle]
\label{P_wd}
The weighted Delannoy numbers  are given by 
$$
d_{m,n}^{u,v,w}=\sum_{k=0}^{n}  \binom{m+n-k}{k}\binom{m+n-2k}{n-k} u^{m-k}
v^{n-k} w^k. 
$$
\end{proposition}
Note that Equation (\ref{E_FR}) and Proposition~\ref{P_wd} have the following
one-line proof: a Delannoy path from $(0,0)$ to $(m,n)$ containing $k$
northeast steps must contain $(m-k)$ east and $(n-k)$ north steps, and
these steps may be listed in $\binom{m+n-k}{k,m-k,n-k}$ ways.

Setting $m=n$ in Proposition~\ref{P_wd} yields the following formula for
the weighted central Delannoy numbers:
\begin{equation}
\label{E_wcd}
d_{n,n}^{u,v,w}=\sum_{k=0}^{n}  \binom{2n-k}{k}\binom{2n-2k}{n-k} u^{n-k}
v^{n-k} w^k. 
\end{equation}
Using (\ref{E_sP}) and (\ref{E_wcd}) we may show the following formula.
\begin{lemma}
\label{L_wcd}
The weighted central Delannoy numbers are linked to the shifted Legendre
polynomials by 
$$
d_{n,n}^{u,v,w}=(-w)^n \sP_n\left(-\frac{uv}{w}\right).
$$
\end{lemma}
\begin{proof}
In a Delannoy path from $(0,0)$ to $(n,n)$ the total number of north
and northeast steps is $n$. Thus we have 
$$
d_{n,n}^{u,v,w}=(-w)^n d_{n,n}^{u,-v/w,-1}
$$
which is equal to $(-w)^n d_{n,n}^{1,-uv/w,-1}$ since the number of east
steps is the same as the number of north steps.
By (\ref{E_wcd}) we have
$$
d_{n,n}^{1,-uv/w,-1}=\sum_{k=0}^{n}  \binom{2n-k}{k}\binom{2n-2k}{n-k} 
\left(-\frac{uv}{w}\right)^{n-k} (-1)^k,
$$
and the statement follows from (\ref{E_sP}), after replacing $k$ with $n-k$.
\end{proof} 
Substituting $u=v=w=1$ into Lemma~\ref{L_wcd} yields
$$
d_{n,n}=d_{n,n}^{1,1,1}=(-1)^n \sP _n(-1)=(-1)^n\cdot P_n(-3).
$$
Equation (\ref{E_dp}) now follows from the ``swapping rule''
(\ref{E_pswap}).  
 We may use the same rule to rewrite Lemma~\ref{L_wcd} as follows. 
\begin{proposition}
\label{P_wcdp}
The weighted central Delannoy numbers and the shifted Legendre
polynomials satisfy
$$
d_{n,n}^{u,v,w}=w^n \sP_n\left(\frac{uv}{w}+1\right).
$$
\end{proposition} 
\begin{proof}
By definition, we have
$(-1)^n \sP_n (-x)=(-1)^n P_n(-2x-1)$.
Using (\ref{E_pswap}) we obtain
$$
(-1)^n P_n(-2x-1)=P_n(2x+1)=\sP _n(x+1).
$$
Thus we may rewrite (\ref{E_pswap}) for shifted Legendre polynomials as
\begin{equation}
\label{E_spswap}
(-1)^n
\sP_n(-x)=\sP_n(x+1).
\end{equation}
The statement now follows immediately from Lemma~\ref{L_wcd}.
\end{proof}
\begin{remark}
Using Lemma~\ref{L_wcd} and Proposition~\ref{P_wcdp} we obtain two
infinite sets of weightings yielding the central Delannoy numbers
$d_{n,n}$ as the total weight of all Delannoy paths from $(0,0)$ to
$(n,n)$. Indeed, Lemma~\ref{L_wcd} implies
$$
d_{n,n}=d_{n,n}^{r,1/r,1} \quad\mbox{for all $r\in{\mathbb R}\setminus
  \{0\}$,} 
$$
whereas Proposition~\ref{P_wcdp} yields
$$
d_{n,n}=d_{n,n}^{r,2/r,-1} \quad\mbox{for all $r\in{\mathbb R}\setminus
  \{0\}$.} 
$$
\end{remark}
In order to extend the validity of our formulas to non-central
weighted Delannoy numbers, we prove the following generalization of
Equation (\ref{E_sP}).
\begin{proposition}
\label{P_sj}
For all $\alpha\in {\mathbb N}$ and $\beta\in {\mathbb C}$ we have 
$$
(x-1)^{\alpha}\sP^{(\alpha,\beta)}_n(x)=
\sum_{k=0}^{n+\alpha} (-1)^{n+\alpha-k} x^k \binom{n+\alpha}{k}
\binom{n+\beta+k}{n}.
$$
\end{proposition} 
\begin{proof}
Using (\ref{E_J1}) we may write
\begin{align*}
(x-1)^{\alpha}\sP^{(\alpha,\beta)}_n(x)&=
\sum_{j=0}^n\binom{n+\alpha+\beta+j}{j}\binom{n+\alpha}{n-j} (x-1)^{\alpha+j}\\
&=
\sum_{j=0}^n\binom{n+\alpha+\beta+j}{j}\binom{n+\alpha}{\alpha+j}
\sum_{k=0}^{\alpha+j} 
(-1)^{\alpha+j-k}\binom{\alpha+j}{k}x^k.\\
\end{align*}
Since 
$$\binom{n+\alpha}{\alpha+j}\cdot \binom{\alpha+j}{k}
=
\binom{n+\alpha}{k}\cdot \binom{n+\alpha-k}{\alpha+j-k}
=
\binom{n+\alpha}{k}\cdot \binom{n+\alpha-k}{n-j}, 
$$
changing the order of summation in the previous equation yields
$$
(x-1)^{\alpha}\sP^{(\alpha,\beta)}_n(x)=
\sum_{k=0}^{n+\alpha} x^k \binom{n+\alpha}{k}
\sum_{j=0}^n
(-1)^{\alpha+j-k}\binom{n+\alpha+\beta+j}{j}\binom{n+\alpha-k}{n-j}.
$$
It should be noted that $\binom{n+\alpha-k}{n-j}=0$ if $\alpha<k$ and
$j<k-\alpha$. Since 
$$
(-1)^j \binom{n+\alpha+\beta+j}{j}
=\binom{-n-\alpha-\beta-1}{j},
$$
the last equation may be rewritten as 
$$
(x-1)^{\alpha}\sP^{(\alpha,\beta)}_n(x)=
\sum_{k=0}^{n+\alpha} (-1)^{\alpha-k} x^k \binom{n+\alpha}{k}
\sum_{j=0}^n
\binom{-n-\alpha-\beta-1}{j}\binom{n+\alpha-k}{n-j}.
$$
Using the well-known polynomial identity 
$$
\sum_{j=0}^n \binom{X}{j}\binom{Y}{n-j}=\binom{X+Y}{n}
$$
we obtain
$$
\sum_{j=0}^n
\binom{-n-\alpha-\beta-1}{j}\binom{n+\alpha-k}{n-j}=\binom{-\beta-k-1}{n}.
$$
The statement now follows from the last equation for
$(x-1)^{\alpha}\sP^{(\alpha,\beta)}_n(x)$ and from
$$
\binom{-\beta-k-1}{n}=(-1)^n\binom{\beta+k+n}{n}.
$$

\end{proof}
\begin{corollary}
\label{C_sj}
In the case when $\alpha=0$ we obtain
$$
\sP^{(0,\beta)}_n(x)=
\sum_{k=0}^{n} (-1)^{n-k} x^k \binom{n}{k}
\binom{n+\beta+k}{n}.
$$
\end{corollary}
Using Corollary~\ref{C_sj} and Proposition~\ref{P_wd} we may 
generalize Lemma~\ref{L_wcd} to the main result of this section.
\begin{theorem}
\label{T_wd}
The weighted Delannoy numbers and the shifted Jacobi polynomials are
linked by the formula
$$
d_{n+\beta,n}^{u,v,w}=u^{\beta}(-w)^n
\sP^{(0,\beta)}_n\left(-\frac{uv}{w}\right).    
$$
Here $\beta\in{\mathbb Z}$ is any integer satisfying $\beta\geq -n$.
\end{theorem}
\begin{proof}
By Proposition~\ref{P_wd} we have
$$
d_{n+\beta,n}^{u,v,w}=\sum_{k=0}^{n}
\binom{2n+\beta-k}{k}\binom{2n+\beta-2k}{n-k} 
u^{n+\beta-k}v^{n-k} w^k. 
$$
Replacing $k$ with $(n-k)$ yields
$$
d_{n+\beta,n}^{u,v,w}=\sum_{k=0}^{n}
\binom{n+\beta+k}{n-k}\binom{\beta+2k}{k} 
u^{\beta+k}v^{k} w^{n-k}. 
$$
Here 
$$
\binom{n+\beta+k}{n-k}\binom{\beta+2k}{k} 
=
\binom{n+\beta+k}{n-k,\beta+k,k} 
=
\binom{n+\beta+k}{n}\binom{n}{k},
$$
thus we also have 
$$
d_{n+\beta,n}^{u,v,w}=\sum_{k=0}^{n}
\binom{n+\beta+k}{n}\binom{n}{k} 
u^{\beta+k}v^{k} w^{n-k}. 
$$
On the other hand, by Corollary~\ref{C_sj} we have
$$
(-w)^n\sP^{(0,\beta)}_n\left(-\frac{uv}{w}\right)=
\sum_{k=0}^{n} \binom{n}{k}
\binom{n+\beta+k}{n} u^k v^k w^{n-k}, 
$$
which differs from $d_{n+\beta,n}^{u,v,w}$ only by a factor of $u^{\beta}$.
\end{proof}
Direct substitution of $u=v=w=1$ into Theorem~\ref{T_wd} yields
\begin{equation}
d_{n+\beta,n}=(-1)^n \sP^{(0,\beta)}_n(-1)=(-1)^n P^{(0,\beta)}_n(-3),
\end{equation}
a simpler formula may be obtained by using the ``swapping rule''
(\ref{E_jswap}), which, after replacing the letter $\beta$ with
$\alpha$, yields (\ref{E_dp1}). 
We conclude this section with the generalization of
Proposition~\ref{P_wcdp} to all weighted Delannoy numbers.
\begin{proposition}
\label{P_wdp}
The weighted Delannoy numbers and the shifted 
Jacobi polynomials satisfy
$$
d_{n+\beta,n}^{u,v,w}=u^{\beta}w^n
\sP_n^{(\beta,0)}\left(\frac{uv}{w}+1\right). 
$$
\end{proposition} 
\begin{proof}
In analogy to the argument seen in the proof of
Proposition~\ref{P_wcdp}, we may rewrite (\ref{E_jswap}) for shifted
Jacobi polynomials as 
\begin{equation}
\label{E_sjswap}
(-1)^n \sP^{(\alpha,\beta)}_n(-x)=\sP^{(\beta,\alpha)}_n(x+1).
\end{equation}
The statement is an immediate consequence of this 
swapping rule and Theorem~\ref{T_wd}.
\end{proof}

\section{Orthogonality of Jacobi polynomials with
  natural number parameters}
\label{s_orth}

As a consequence of Theorem~\ref{T_wd} we obtain the following 
lattice path representations of the shifted Legendre polynomials
$\sP_n(x)=\sP_n^{(0,0)}(x)$ and the shifted Jacobi polynomials
$\sP_n^{(0,\beta)}(x)$. 

\begin{corollary}
\label{C_llp}
For all $\beta\in{\mathbb Z}$ satisfying $\beta\geq -n$, the shifted
Jacobi polynomial 
$\sP_n^{(0,\beta)}(x)$ is $d_{n+\beta,n}^{1,x,-1}$, i.e., the total weight
of all Delannoy paths from $(0,0)$ to $(n+\beta,n)$, where each east
step has weight $1$, each north step has weight $x$, and each northeast
step has weight $-1$. 
\end{corollary} 

In this section we use this Corollary to provide a combinatorial proof 
for all $\alpha,\beta\in{\mathbb N}$, of the
fact that the Jacobi polynomials $P_n^{(\alpha,\beta)}(x)$ form an
orthogonal basis with respect to the inner product   
$$\langle f, g\rangle :=\int_{-1}^1 f(x)\cdot g(x)\cdot
(1-x)^{\alpha}(1+x)^{\beta}\ dx.$$
Using the substitution $u:=(x+1)/2$ yields 
$$
\int_{-1}^1 f(x)\cdot g(x)\cdot (1-x)^{\alpha}(1+x)^{\beta}\ dx
=
2^{\alpha+\beta+1}\int_{0}^1 f(2u-1)\cdot g(2u-1)\cdot
(1-u)^{\alpha}u^{\beta}\ du 
$$  
thus the orthogonality relations we want to prove may be restated for
shifted Jacobi (and Legendre) polynomials as follows.
\begin{theorem}
\label{T_orth}
For all $\alpha,\beta\in{\mathbb N}$, the shifted Jacobi polynomials
$\sP_n^{(\alpha,\beta)}(x)$ form an
orthogonal basis with respect to the inner product   
$$\langle f, g\rangle :=\int_{0}^1 f(x)\cdot g(x)\cdot
(1-x)^{\alpha}x^{\beta}\ dx.$$
\end{theorem}
First we prove Theorem~\ref{T_orth} for the case $\alpha=0$ only. Since 
an inner product of any polynomial with itself is nonzero, it is
sufficient to show the following orthogonality relation.
\begin{proposition}
\label{P_orth}
For all $m,n,\beta\in {\mathbb N}$ satisfying $m<n$ we have
$$\int_{0}^1 x^{m+\beta}\cdot \sP^{(0,\beta)}_n(x)\ dx=0.$$
\end{proposition}
\begin{proof}
Using the lattice path representation stated in Corollary~\ref{C_llp},
each Delannoy path from $(0,0)$ to $(n+\beta,n)$ containing $k$ 
north steps contributes a term $(-1)^{n-k} x^k$ to
$\sP^{(0,\beta)}_n(x)$, and a term $(-1)^{n-k} x^{k+m+\beta}$ to
$x^{m+\beta}\cdot\sP^{(0,\beta)}_n(x)$. Since
$$
\int_{0}^1 x^{k+m+\beta}\ dx=\frac{1}{k+m+\beta+1},
$$
we obtain that each Delannoy path from $(0,0)$ to $(n+\beta,n)$
containing $k$ north steps contributes a term
$$
\frac{(-1)^{n-k} (n+m+\beta+1)!}{k+m+\beta+1}
\quad\mbox{to}\quad 
(n+m+\beta+1)!\cdot \int_{0}^1 x^{m+\beta}\cdot \sP^{(0,\beta)}_n(x)\ dx.
$$
Therefore 
$$
(n+m+\beta+1)!\cdot \int_{0}^1 x^{m+\beta}\cdot \sP^{(0,\beta)}_n(x)\ dx
$$
is the total weight of all pairs $(L,\sigma)$ where $L$ is a Delannoy
path from $(0,0)$ to $(n+\beta,n)$ and $\sigma$ is a bijection 
$\{r,a_1,\ldots,a_{n+\beta},b_1,\ldots,b_m\}\rightarrow
\{1,\ldots,m+n+\beta+1\}$, subject to the following rules:
\begin{itemize}
\item[(i)] $\sigma(r)<\sigma(a_i)$ holds for all $i$ such that 
there is an east step in $L$ from $(i-1,y)$ to $(i,y)$ for some $y$;
\item[(ii)] $\sigma(r)<\sigma(b_j)$ holds for $j=1,2,\ldots,m$.
\end{itemize}
The weight of the pair $(L,\sigma)$ is defined to be a function of the
Delannoy path $L$: each northeast step contributes a factor of $(-1)$, all
other steps contribute a factor of $1$.

Indeed, a Delannoy path $L$ from $(0,0)$ to $(n+\beta,n)$ with
$k$ north steps has $n-k$ northeast steps and
$n+\beta-(n-k)=\beta+k$ east steps. After fixing $L$, condition (i)
above requires $\sigma(r)$ to be less than $\beta+k$ elements in
$\{\sigma(a_1),\ldots,\sigma(a_{n+\beta})\}$. Conditions (i) and (ii)
together force exactly $\beta+k+m$ elements of
$\{1,\ldots,m+n+\beta+1\}$ to be more than $\sigma(r)$. Thus the number
of bijections $\sigma$ forming a valid pair $(L,\sigma)$ with this fixed $L$ is
$(n+m+\beta+1)!/(k+m+\beta+1)$ and the total contribution of all pairs
$(L,\sigma)$ for this fixed $L$ is $(-1)^{n-k} (n+m+\beta+1)!/(k+m+\beta+1)$.
Let us call the pairs $(L,\sigma)$ satisfying (i) and (ii) {\em valid
  pairs}. 

Let us now eliminate the contribution  of some valid pairs $(L,\sigma)$ by
introducing the following involutions $\tau_i$ for $i=1,2,\ldots,n+\beta$. 
If the Delannoy path $L$ contains a northeast step from some $(i-1,y)$
to $(i,y+1)$ and $\sigma(r)<\sigma(a_i)$ then 
we define $\tau_i(L,\sigma)$ as $(L',\sigma)$ where $L'$ is obtained
from $L$ by replacing the northeast step from $(i-1,y)$ to $(i,y+1)$  
with an east step from $(i-1,y)$ to $(i,y)$ followed by a north step 
from $(i,y)$ to $(i,y+1)$. (Note that the condition
$\sigma(r)<\sigma(a_i)$ guarantees that $(L',\sigma)$ is a valid pair.) 
If $L$ contains an east step from $(i-1,y)$ to $(i,y)$ followed by a
north step from $(i,y)$ to $(i,y+1)$, we define
$\tau_i(L,\sigma)$ as $(L',\sigma)$ where $L'$ is obtained
from $L$ by replacing the sequence of one east step and one north step from 
$(i,y)$ to $(i,y+1)$ by a diagonal step from $(i,y)$ to $(i,y+1)$. 
(Note that condition (i) for $(L,\sigma)$ requires
$\sigma(r)<\sigma(a_i)$ in this case.)  In all
other situations we define $\tau_i(L,\sigma):=(L,\sigma)$. 
Obviously, the $\tau_i$'s are involutions on the set of valid pairs, and they
commute pairwise, since they change disjoint parts of the underlying
Delannoy paths. Thus they induce a ${\mathbb Z}_2^{n+\beta}$-action on the
set of  valid pairs. If $\tau_i(L,\sigma)\neq (L,\sigma)$ then the weight
of $\tau_i(L,\sigma)$ is the negative of the weight of $(L,\sigma)$. Thus
the weight of all $(L,\sigma)$ pairs that belong to the same ${\mathbb
  Z}_2^{n+\beta}$-orbit cancels, unless the orbit in question contains a
single fixed point. 

We have shown that 
$$
(n+m+\beta+1)!\cdot \int_{0}^1 x^{m+\beta}\cdot \sP^{(0,\beta)}_n(x)\ dx
$$
equals the total weight of all valid pairs $(L,\sigma)$ satisfying 
$\tau_i(L,\sigma)= (L,\sigma)$ for $i=1,2,\ldots,n+\beta$. 
These pairs may be characterized by the following two conditions:
\begin{itemize}
\item[(iii)] $\sigma(r)>\sigma(a_i)$ holds for all $i$ such that 
there is a northeast east step in $L$ from $(i-1,y)$ to $(i,y+1)$ for
some $y$;
\item[(iv)] there is no east step immediately followed by a north step. 
\end{itemize}
Condition (iv) is a statement about the Delannoy path $L$ only. It
may be stated equivalently by requiring that the only way to go before
the first, after the last, or  between two consecutive northeast steps is 
to use all the north steps before all the east steps. Such a Delannoy
path is uniquely determined by the set of its diagonal steps. A sequence of 
diagonal steps can be completed to a Delannoy path satisfying (iv) if
and only if the first coordinates and the second coordinates of the
starting points both form a strictly increasing sequence. Thus there are 
$\binom{n+\beta}{k}\binom{n}{k}$ Delannoy paths satisfying (iv). 
Given a Delannoy path $L$ satisfying (iv) having $k$ diagonal steps,
conditions (i), (ii), and (iii) are equivalent to stating that
$\sigma(r)=k+1$, $\sigma(a_i)\leq k$ for all $i$ such that 
there is a northeast east step in $L$ from $(i-1,y)$ to $(i,y+1)$ for
some $y+1$, and the image of all remaining elements under $\sigma$
belongs to $\{k+1,k+2,\ldots,n+m+\beta+1\}$. There are
$k!(n+m+\beta-k)!$ ways to find such a $\sigma$. 

Therefore we obtain the following equality:
\begin{equation}
\label{E_PL}
(n+m+\beta+1)!\cdot \int_{0}^1 x^{m+\beta}\cdot \sP^{(0,\beta)}_n(x)\ dx
=
\sum_{k=0}^n (-1)^k \binom{n+\beta}{k}\binom{n}{k}\cdot k!(n+m+\beta-k)!
\end{equation}   
We are left to show that the right hand side is zero for $m<n$. This is
known in the theory of rook polynomials (see Section~\ref{s_rook}), so
here we indicate a short proof for completeness sake only.
The right hand side is
$(m+\beta)!$ times 
$$
p(m):=\sum_{k=0}^n (-1)^k \binom{n}{k}(n+\beta-k+1)_k (m+\beta)_{n-k}
=(-1)^n \sum_{k=0}^n  \binom{n}{k}(n+\beta-k+1)_k (-m-\beta-n+k)_{n-k}.$$
Consider $p(m)$ as a polynomial function of $m$. The number $(-1)^n p(-m)$
is then the number of ways to select a $k$-element subset of an
$n$-element set and injectively color its elements using $n+\beta$
colors, then color the remaining $n-k$ elements injectively, using a
disjoint set of $m-\beta-1$ colors. Thus we have
$$
(-1)^n p(-m)=\binom{n+m-1}{n}\quad\mbox{implying}\quad 
p(m)=(-1)^n\binom{n-m-1}{n}.  
$$
This is zero for $m<n$. 
\end{proof}
The rest of the proof of Theorem~\ref{T_orth} is a consequence of the
following Proposition. 
\begin{proposition}
\label{P_abdec}
For all $\alpha,\beta\in {\mathbb N}$, the shifted Jacobi polynomials
satisfy
$$
(x-1)^{\alpha}\sP^{(\alpha,\beta)}_n(x)
=
\sum_{i=0}^{\alpha}\binom{\alpha}{i} 
(-1)^i x^{\alpha-i} \sP_{n}^{(0,\alpha+\beta-i)}(x). 
$$
\end{proposition}
\begin{proof}
First we show using Proposition~\ref{P_sj} that
$(x-1)^{\alpha}\sP^{(\alpha,\beta)}_n(x)$ is the total weight of all
Delannoy paths from $(0,0)$ to $(n+\alpha+\beta,n+\alpha)$ subject to
the restriction that none of the first $\alpha$ steps is an east
step. Here each east step contributes a factor of $1$, each north step
contributes a factor of $x$ and each northeast step contributes a factor
of $-1$. Indeed, if such a Delannoy path has $k$ north steps then it
has $n+\alpha-k$ northeast steps and $\beta+k$ east steps. There are 
$\binom{n+\alpha}{k}$ ways to determine the order of the north and the
northeast steps among themselves and then there are $\binom{n+\beta}{n}$
ways to determine the order of the $\beta+k$ east steps with respect to the
$n$ other steps (we subtracted $\alpha$ from $n+\alpha$ for the first
$\alpha$ steps which can not be east steps).

Every Delannoy path subject to the above restriction may be uniquely
decomposed as follows. The first $\alpha$ steps form a Delannoy path 
with no east step from $(0,0)$ to some $(i,\alpha)$ where $0\leq i\leq
\alpha$ is the number of northeast steps among the first $\alpha$
steps. These contribute a factor of $(-1)^i x^{\alpha-i}$. Given $i$, there are
$\binom{\alpha}{i}$ ways to determine of the order of the first $\alpha$
steps. The rest of the path is an unrestricted Delannoy path from
$(i,\alpha)$ to $(n+\alpha+\beta,n+\alpha)$, shifting its start 
to the origin yields a Delannoy path from $(0,0)$ to
$(n+\alpha+\beta-i,n)$. This part of the path may be chosen
independently among all Delannoy paths from $(0,0)$ to
$(n+\alpha+\beta-i,n)$ and, by Corollary~\ref{C_llp}, their total
contribution is $\sP_{n}^{(0,\alpha+\beta-i)}(x)$.
\end{proof}

Using Proposition~\ref{P_abdec} we may conclude the proof of
Theorem~\ref{T_orth} as follows. Assume $\alpha,\beta\in {\mathbb N}$
and $m<n$. Then
\begin{align*}
\int_{0}^1 x^m\cdot \sP^{(\alpha,\beta)}_n(x)\cdot
(1-x)^{\alpha}x^{\beta}\ dx
&=
(-1)^{\alpha}
\int_{0}^1 x^{m+\beta}\cdot\sum_{i=0}^{\alpha}\binom{\alpha}{i} (-1)^i
x^{\alpha-i} \sP_{n}^{(0,\alpha+\beta-i)}(x) \ dx\\
&= \sum_{i=0}^{\alpha}\binom{\alpha}{i} (-1)^{\alpha+i} \int_{0}^1
x^{m+(\alpha+\beta-i)}\cdot \sP_{n}^{(0,\alpha+\beta-i)}(x) \ dx.
\end{align*}
All integrals in the last sum are zero by Proposition~\ref{P_orth}.  

\section{Connections to the classical theory of rook polynomials}
\label{s_rook}

For $\beta=0$ equation~(\ref{E_PL}) takes the form
\begin{equation}
\label{E_PL0}
(n+m+1)!\cdot \int_{0}^1 x^{m}\cdot \sP_n(x)\ dx
=
\sum_{k=0}^n (-1)^k \binom{n}{k}^2\cdot k!(n+m-k)!
\end{equation}   
The right hand side here is ${\mathcal L}(x^m l_n(x))$ where $l_n(x)$ is the
Laguerre polynomial defined in (\ref{E_Ldef}). The orthogonality of
these Laguerre polynomials with respect to the inner product induced by
${\mathcal L}$ is equivalent to stating
$$
{\mathcal L}(x^m l_n(x))=\int_0^{\infty} x^m l_n(x) e^{-x}\ dx=0 \quad\mbox{for $m<n$}.
$$
Thus the combinatorial orthogonality of the shifted Legendre polynomials 
reduces to the combinatorial orthogonality of the Laguerre polynomials
used in the theory of rook polynomials. At a hasty look, the connection
is ``almost obvious'' visually, since in both combinatorial situations the term
$\binom{n}{k}^2$ expresses the number of ways to select $k$ rows and $k$
columns on some $n\times n$ ``chess-board'', as the rows and columns where
diagonal steps, respectively, rooks are located. The main difference is
in the next step of the selection process: in the case of the valid pairs
$(L,\sigma)$ that are fixed points of the involutions $\tau_i$ the
diagonal steps must follow a strictly increasing order in the picture,
and a factor of $k!$ arises by selecting the restriction of an injective
map $\sigma$ to the columns associated to the diagonal steps. In rook
polynomial theory, however, the selected rooks may have any permutation
pattern, thus a factor of $k!$ arises in a completely different way.
It should be also noted that the condition $m<n$ was not used in the
proof of (\ref{E_PL}), thus we may state in general
\begin{equation}
(n+m+1)!\cdot \int_{0}^1 x^{m}\cdot \sP_n(x)\ dx=\int_0^{\infty} x^m
  l_n(x) e^{-x}\ dx \quad\mbox{for all $m,n\in{\mathbb N}$.} 
\end{equation}
For general $\beta\in{\mathbb N}$, equation~(\ref{E_PL}) may be
rewritten as 
\begin{equation}
\label{E_PL2}
(n+m+\beta+1)!\cdot \int_{0}^1 x^{m+\beta}\cdot \sP^{(0,\beta)}_n(x)\ dx
=
\int_0^{\infty} x^m l_n^{(\beta)}(x) x^{\beta} e^{-x}\ dx \quad\mbox{for
    all $m,n\in{\mathbb N}$.} 
\end{equation}   
Here 
$$
l_n^{(\beta)}(x):=\sum_{k=0}^n (-1)^k \binom{n+\beta}{k}\binom{n}{k}
k! x^{n-k} 
$$
is the $n$-th {\em generalized Laguerre polynomial} associated to the
rectangular board $[n+\beta]\times [n]$ (as a subset of itself). These
generalized Laguerre polynomials are related to the usually normalized
generalized Laguerre polynomials $L^{(\beta)}_n(x)$ by the formula
$$
l_n^{(\beta)}(x)=(-1)^n n! L^{(\beta)}_n(x), 
$$
and form an orthogonal basis with respect to the weight function
$x^{\beta}e^{-x}$. 
\begin{remark}
The reader familiar with the work of Foata and
Zeilberger~\cite{Foata-Zeilberger} should also note that in their work too,
showing the orthogonality of the Jacobi polynomials is in a way reduced to the
study of a combinatorial model for the Laguerre
polynomials (presented in~\cite{Foata-Zeilberger-L}). 
The question naturally arises
whether a closer relation could be found between the orthogonality proof
of Foata and Zeilberger~\cite{Foata-Zeilberger} and the weighted lattice path
enumeration proof presented in this paper. Finding such a connection 
depends on surmounting the following difficulty: in the work of Foata and
Zeilberger the parameters $\alpha$ and $\beta$ are weights (associated
to bipermutations) and left as variables, whereas in the present model
$\alpha$ and $\beta$ need to be a natural numbers, since the coordinates
of the endpoints of the lattice paths considered are expressed in terms
of these parameters. 
\end{remark}

\section{Shifted Jacobi polynomials with zero first parameter and negative
  integer second parameter} 
\label{s_bnneg}

Using weighted Delannoy numbers it is easy to connect most shifted Jacobi
polynomials $\sP_n^{(0,\beta)}(x)$ with negative integer $\beta$ to the
Jacobi polynomials with positive integer $\beta$. Indeed, we have the
following formula.
\begin{proposition}
\label{P_bneg}
For $\beta\in{\mathbb N}$ and $n\geq \beta$ we have 
$$
\sP_n^{(0,-\beta)}(x)=x^{\beta} \sP_{n-\beta}(x).
$$
\end{proposition}
\begin{proof}
By Corollary~\ref{C_llp} we have
$$
\sP_n^{(0,\beta)}(x)=d_{n+\beta,n}^{1,x,-1}.
$$ 
Dividing the weight of all north and northeast steps by $x$ we may take
out a factor of $x^n$ and get
$$
\sP_n^{(0,\beta)}(x)=x^n d_{n+\beta,n}^{1,1,-1/x}.
$$  
Since each east and each north step has the same weight, we may swap the
horizontal and vertical axis and get
$$
\sP_n^{(0,\beta)}(x)=x^n d_{n,n+\beta}^{1,1,-1/x}.
$$  
Finally, multiplying both sides by $x^{\beta}$ and using the arising
factor of $x^{n+\beta}$ on the right hand side to multiply the weight of
each north and each northeast step by $x$ yields
$$
x^{\beta} \sP_n^{(0,\beta)}(x)=d_{n,n+\beta}^{1,x,-1}.
$$  
Here the right hand side is equal to $\sP_{n+\beta}^{(0,\beta)}(x)$ 
by Corollary~\ref{C_llp}. The statement follows by replacing $n$ with
$n-\beta$. 
\end{proof}
Proposition~\ref{P_bneg} does not apply to $\sP^{(0,-\beta)}_n(x)$ when 
$n<\beta$. Table~\ref{t_bneg} contains these polynomials and
$\sP^{(0,-6)}_6(x)$ for $\beta=6$.

\begin{table}[h]
\begin{tabular}{c||r|r|r|r|r|r||r}
$n$&$0$&$1$&$2$&$3$&$4$&$5$&$6$\\
\hline
$\sP^{(0,-6)}_n(x)$&$1$&$5-4x$&$3x^2-12x+10$&$3x^2-12x+10$&$5-4x$&$1$&$x^6$\\
\end{tabular}
\caption{$\sP^{(0,-6)}_n(x)$ for $n\leq 6$}
\label{t_bneg}
\end{table}

The list of polynomials exhibits a certain symmetry which is most easily
shown for the transformed Jacobi polynomials 
$P^{(\alpha,\beta)}_n(2x+1)$. By \eqref{eq:ssPP}, for the values of $n$
allowed by Theorem~\ref{thm:Askey-Romanovski}, these polynomials 
are equal to the Romanovski polynomials $\ssP^{(\alpha,\beta)}_n(x)$.
\begin{definition}
We extend the definition of the Romanovski polynomials
$\ssP^{(\alpha,\beta)}_n(x)$ to all values of $n$ using \eqref{eq:ssPP}.
\end{definition}
As an immediate consequence of (\ref{E_J1}) and \eqref{eq:ssPP} we
obtain the following. 
\begin{equation}
\label{E_ssP}
\ssP_n^{(\alpha,\beta)}(x)
=\sum_{j=0}^n\binom{n+\alpha+\beta+j}{j}\binom{n+\alpha}{n-j} 
x^j.
\end{equation}
In particular, substituting $\alpha=0$ and $-\beta$ for $\beta$ yields
$$
\ssP_n^{(0,-\beta)}(x)
=\sum_{j=0}^n\binom{n-\beta+j}{j}\binom{n}{j} x^j,
$$
which may be rewritten as
\begin{equation}
\label{E_ssP0}
\ssP_n^{(0,-\beta)}(x)
=\sum_{j=0}^n\binom{\beta-1-n}{j}\binom{n}{j} (-x)^j.
\end{equation}
\begin{corollary}
\label{C_bneg}
For $\beta \in {\mathbb N}$ and $0\leq n\leq \beta-1$ we have
$$
\ssP^{(0,-\beta)}_n(x)=\ssP^{(0,-\beta)}_{\beta-1-n}(x),
$$
implying also 
$$
P^{(0,-\beta)}_n(x)=P^{(0,-\beta)}_{\beta-1-n}(x)\quad \mbox{and}\quad
\sP^{(0,-\beta)}_n(x)=\sP^{(0,-\beta)}_{\beta-1-n}(x).
$$
\end{corollary}
Because of Corollary~\ref{C_bneg}, the Jacobi polynomials
$\{P^{(0,-\beta)}_n(x)\}_{n\geq 0}$ or, equivalently,  
the polynomials $\{\ssP^{(0,-\beta)}_n(x)\}_{n\geq 0}$
can not form an infinite orthogonal
polynomial sequence. There is a chance to having such a sequence only up to
$n=\lfloor(\beta-1)/2\rfloor$. 
Theorem~\ref{thm:Askey-Romanovski} guarantees having a finite sequence
of orthogonal polynomials up to {\em almost} that degree. 
Indeed, substituting $\alpha=0$ and replacing $\beta$ with $-\beta$ in 
Theorem~\ref{thm:Askey-Romanovski}, the inequality 
$\alpha+\beta+N+1<0$ becomes $N<\beta-1$ which, for integers, is
equivalent to $N\leq \beta-2$. Thus, Theorem~\ref{thm:Askey-Romanovski}
states the orthogonality of the polynomials $\ssP^{(0,-\beta)}_n(x)$ of
degree at most $(\beta-2)/2$ only. 
\begin{corollary}[Askey-Romanovski]
\label{cor:Askey-Romanovski}
For any positive integer $\beta\geq 2$, the polynomials
$\{\ssP^{(0,-\beta)}_n(x)\::\: n\leq \lfloor(\beta-2)/2\rfloor\}$
form a finite sequence of orthogonal polynomials.  
\end{corollary}
We will provide a combinatorial proof of a strengthened version 
of Corollary~\ref{cor:Askey-Romanovski} that, for odd values of $\beta$,
allows adding the polynomial $\ssP^{(0,-\beta)}_{(\beta-1)/2}(x)$ to our
finite sequence of orthogonal polynomials, thus filling the only
remaining gap. Note that the linear functional 
\begin{equation}
\label{eq:lbeta}
{\mathcal L}_{\beta}(x^k):=\int_0^{\infty} x^k(1+x)^{-\beta}\ dx
\quad\mbox{for $0\leq k\leq \beta-2$,}
\end{equation}
implicitly used in Theorem~\ref{thm:Askey-Romanovski}, is by
definition~\cite[6.2.1]{Abramowitz-Stegun} an evaluation of the {\em beta
function} $B(z,w)$. Thus we obtain 
\begin{equation}
\label{eq:lbeta2}
{\mathcal L}_{\beta}(x^k)=B(k+1,\beta-1-k)=
\int_0^{1}
\left(\frac{t}{1-t}\right)^k (1-t)^{\beta-2}\ dt.
\end{equation}
Since $B(z,w)=\Gamma(z)\Gamma(w)/\Gamma(z+w)$
(see~\cite[6.2.2]{Abramowitz-Stegun})
and $\Gamma(z+1)=z!$ holds for all $z\in{\mathbb N}$ (see
by~\cite[6.1.15]{Abramowitz-Stegun}), we also have  
\begin{equation}
\label{eq:lbeta3}
{\mathcal L}_{\beta}(x^k)=\frac{k!\cdot (\beta-2-k)!}{(\beta-1)!}
\quad\mbox{for $0\leq k\leq \beta-2$.} 
\end{equation}
Keeping these formulas in mind, we are ready to present a combinatorial
proof of our extended version of
Corollary~\ref{cor:Askey-Romanovski}.
\begin{theorem}
\label{T_bneg}
Let $\beta\geq 2$ be any positive integer and let ${\mathcal L}_{\beta}$ be the
linear functional defined on the  vector space 
$\{p(x)\in {\mathbb C}[x]\::\: \deg (p)\leq \beta-2\}$ by \eqref{eq:lbeta}.
Then the Romanovski polynomials 
$\{\ssP^{(0,-\beta)}_n(x)\::\: 0\leq n\leq (\beta-2)/2\}$
form an orthogonal basis in the with respect to inner product 
$\langle f, g\rangle :={\mathcal L}_{\beta}(f\cdot g)$ on 
the vector space $\{p(x)\in {\mathbb C}[x]\::\: \deg (p)\leq (\beta-2)/2\}$.
For odd $\beta$ we may extend ${\mathcal L}_{\beta}$ and the induced inner
product to polynomials of degree at most
$(\beta-1)/2$ by making ${\mathcal L}_{\beta}(x^{\beta-1})$ large enough to make the
determinant of the $(\beta+1)/2\times (\beta+1)/2$ matrix 
$$
M_{\beta}:=\left(
\begin{array}{cccc}
{\mathcal L}_{\beta}(x^0) &{\mathcal L}_{\beta}(x^1)&\cdots & {\mathcal L}_{\beta}(x^{(\beta-1)/2})\\
{\mathcal L}_{\beta}(x^1) &{\mathcal L}_{\beta}(x^2)&\cdots & {\mathcal L}_{\beta}(x^{(\beta-1)/2+1})\\
\vdots&\vdots& \ddots&\vdots\\
{\mathcal L}_{\beta}(x^{(\beta-1)/2})& {\mathcal L}_{\beta}(x^{(\beta-1)/2+1})&\cdots &{\mathcal L}_{\beta}(x^{\beta-1})\\
\end{array}
\right)
$$
positive. The polynomial  $\ssP^{(0,-\beta)}_{(\beta-1)/2}(x)$ may then
be added to the orthogonal basis. 
\end{theorem}
\begin{proof}
The fact that the bilinear operation $\langle f, g\rangle
:={\mathcal L}_{\beta}(f\cdot g)$ in an inner product is a
consequence of the integral form of ${\mathcal L}_{\beta}(x^k)$ stated in
\eqref{eq:lbeta2}. Using this equation it is easy deduce 
\begin{equation}
\label{E_ip}
\langle f, g\rangle =\int_{0}^1 f\left(\frac{t}{1-t}\right)\cdot
g\left(\frac{t}{1-t}\right)\cdot (1-t)^{\beta-2}\ dt\quad\mbox{if
  $\deg(f)+\deg(g)\leq \beta-2$.}
\end{equation}
Note that the condition $\deg(f)+\deg(g)\leq \beta-2$ implies that the
degree of $(1-t)$ in the denominator of $f(t/(1-t))\cdot
g(t/(1-t))$ is at most $\beta-2$, and so the integrand on the left hand
side is a polynomial function. Thus the bilinear function is in deed an
inner product. To prove the first half of the theorem, it is
sufficient to show 
$$
\langle x^m, \ssP^{(0,-\beta)}_n(x)\rangle
=0\quad\mbox{for $m<n\leq \frac{\beta-2}{2}$}.
$$
By \eqref{eq:lbeta3} this is equivalent to 
$$
\sum_{j=0}^n\binom{\beta-1-n}{j}\binom{n}{j} (-1)^j (m+j)! (\beta-2-m-j)!
=0\quad \mbox{for $m<n\leq \frac{\beta-2}{2}$}.
$$
Dividing both sides by $(\beta-1-n)! m! (n-m-1)!$ yields the equivalent
form
\begin{equation}
\label{E-borth2}
\sum_{j=0}^n (-1)^j \binom{n}{j} \binom{m+j}{m}\binom{\beta-2-m-j}{n-m-1}
=0\quad \mbox{for $m<n\leq \frac{\beta-2}{2}$}.
\end{equation}
The left hand side of (\ref{E-borth2}) is the total weight of all
triplets $(X,A,B)$ where 
\begin{itemize}
\item[(i)] $X$ is a subset of $\{1,2,\ldots,n\}$;
\item[(ii)] $A=\{a_1,\ldots,a_m\}$ is an $m$-element multiset such that
  each $a_i$ belongs to $X\cup\{0\}$; 
\item[(iii)] $B=\{b_1,\ldots,b_{n-m-1}\}$ is an $(n-m-1)$-element
  multiset such that each $b_j$ belongs to
  $\{1,\ldots,\beta-n\}\setminus X$.
\end{itemize}
The weight of the triplet $(X,A,B)$ is $(-1)^{|X|}$. Indeed, for a
fixed $j$ there are $\binom{n}{j}$ ways to select $X$. Once $X$ is
fixed, there are 
$$
\binom{(j+1)+m-1}{m}=\binom{m+j}{m}
$$
ways to select the multiset $A$ on $X\cup \{0\}$ and 
$$
\binom{(\beta-n-j)+(n-m-1)-1}{n-m-1}=\binom{\beta-m-2-j}{n-m-1}
$$
ways to select the multiset $B$ on $\{1,\ldots,\beta-n-1\}\setminus X$. 
Let us now fix the multisets $A$ and $B$ first. Since $|A|+|B|=n-1$,
there is at least one $c\in \{1,\ldots,n\}$ that does not appear in $A$,
nor in $B$. Then, for a subset $X\subset \{1,\ldots,n\}\setminus \{c\}$, the
triplet $(X,A,B)$ satisfies the above conditions if and only if the
triplet $(X\cup\{c\},A,B)$ does. The weight of these triplets cancel
each other.

For the case when $\beta$ is odd, observe first that the above proof of
orthogonality can be adapted to $m<n\leq (\beta-1)/2$ without any
changes since we only need to use the value of ${\mathcal L}_{\beta}(x^k)$ for $k\leq
\beta-2$. The only thing left to show is that we can set the value of
${\mathcal L}_{\beta}(x^{\beta-1})$ in such a way that we get a positive definite inner
product. It is well known that a bilinear form defined by a symmetric real
matrix is positive definite if the principal minors of the matrix are
positive. Thus we only need to make sure that the principal minors of
$M_{\beta}$ are positive. The fact that the restriction of
${\mathcal L}_{\beta}$ to the vector space of polynomials of degree at most $(\beta-3)/2$ is
positive definite implies the positivity of all principal minors of
$M_{\beta}$ except for $\det(M_{\beta})$. Finally, consider
${\mathcal L}_{\beta}(x^{\beta})$ as a variable, and expand $\det(M_{\beta})$ by its last
column. The determinant becomes a linear function of ${\mathcal L}_{\beta}(x^{\beta})$
whose coefficient is the $(\beta-3)/2\times (\beta-3)/2$ principal minor
of $M_{\beta}$, a positive number. Therefore $\det(M_{\beta})>0$ holds
for any sufficiently large ${\mathcal L}_{\beta}(x^{\beta})$. 
\end{proof}

\section{Weighted Schr\"oder numbers and iterated antiderivatives of
  Legendre polynomials}
\label{s_S}

We define a {\em Schr\"oder path} from $(0,0)$ to $(n,n)$ 
as a Delannoy path not going above the line $y=x$. In analogy to
Definition~\ref{D_wdn} we may introduce weighted Schr\"oder numbers as
follows.  
\begin{definition}
\label{D_wsn}
Let $u$, $v$, $w$ be commuting variables. We define the {\em weighted
  Schr\"oder numbers $s_{n}^{u,v,w}$} as the total weight of all
  Schr\"oder paths from $(0,0)$ to $(n,n)$, where each east step $(0,1)$
  has weight $u$, each north step has weight $v$, and each northeast
  step has weight $w$. The weight of a
  lattice path is the product of the weights of its steps. 
\end{definition}
Introducing the {\em Schr\"oder polynomials} by 
\begin{equation}
\label{E_sdef}
S_n(x):=s_{n}^{1,x,-1}\quad\mbox{for $n\geq 0$},
\end{equation}
we have the following formula.
\begin{proposition}
\label{P_SP}
For $n\geq 1$ the Schr\"oder polynomials are given by 
$$
S_n(x)=\frac{x-1}{(n+1)x}\sP^{(1,-1)}_{n}(x).
$$
\end{proposition}
\begin{proof}
First we show that the Schr\"oder polynomials satisfy 
\begin{equation}
\label{E_Sexp}
S_n(x)=\sum_{j=0}^{n}
\frac{(-1)^{n-j}}{j+1}\binom{2j}{j}\binom{n+j}{n-j} x^{j}
\quad\mbox{for $n\geq 1$.} 
\end{equation}
Let $j\in\{0,\ldots,n\}$ be the number of east steps in a Schr\"oder
path, the number of norths steps must be the same, the number of
northeast steps is $(n-j)$. The path is a Schr\"oder path if and only if 
at any given step the number of north steps thus
far does not exceed the number of east steps. The number of ways to arrange $j$
east steps and $j$ north steps satisfying this criterion is the Catalan
number $\binom{2j}{j}/(j+1)$, and there are $\binom{n+j}{n-j}$ ways to
determine the order of the northeast steps with respect to the east and
north steps. 
Since
$$
\binom{2j}{j}\binom{n+j}{n-j}
=\binom{n+j}{j,j,n-j}
=\binom{n+j}{n}\binom{n}{j}\quad\mbox{and}
$$
$$
\frac{1}{j+1}\binom{n}{j}=\frac{1}{n+1}\binom{n+1}{j+1},
$$
we may rewrite (\ref{E_Sexp}) as
\begin{equation}
\label{E_Sexp2}
S_n(x)=\frac{1}{n+1}\sum_{j=0}^{n}
(-1)^{n-j}\binom{n}{j+1}\binom{n+j}{n} x^{j}
\quad\mbox{for $n\geq 1$.} 
\end{equation}
On the other hand, Proposition~\ref{P_sj} gives
$$
(x-1)\sP^{(1,-1)}_{n}(x)=
\sum_{k=0}^{n+1} (-1)^{n+1-k} x^k \binom{n+1}{k}\binom{n-1+k}{n}
\quad\mbox{for $n\geq 1$.} 
$$
Observe that for $n\geq 1$ the constant term in the last sum is zero,
thus we may write 
$$
(x-1)\sP^{(1,-1)}_{n}(x)=
\sum_{k=1}^{n+1} (-1)^{n+1-k} x^k \binom{n}{k}\binom{n-1+k}{n}.
$$
The statement now follows from substituting $j+1$ into $k$ in the last
equation and comparing the result with (\ref{E_Sexp2}). 
\end{proof}
In analogy to Lemma~\ref{L_wcd} we have the following
\begin{lemma}
\label{L_wsn} The weighted Schr\"oder numbers are linked to the
Schr\"oder polynomials by the formula 
$$
s_n^{u,v,w}=(-w)^n S_n\left(-\frac{uv}{w}\right)
$$
\end{lemma}
As a consequence of Proposition~\ref{P_SP} and Lemma~\ref{L_wsn}
we obtain
\begin{corollary}
\label{C_wsn}
For $n\geq 2$, the weighted Schr\"oder numbers are linked to the shifted Jacobi
polynomials by the formula
$$
s_n^{u,v,w}=\frac{(-w)^n
}{n+1}\left(1+\frac{w}{uv}\right)\sP^{(1,-1)}_{n}\left(-\frac{uv}{w}\right). 
$$
\end{corollary}
In particular, substituting $u=v=w=1$ gives that the ordinary Schr\"oder
numbers are given by 
$$
s_n:=s_n^{1,1,1}=\frac{(-1)^n 2}{n+1}\sP^{(1,-1)}_{n}(-1)
=\frac{(-1)^{n}2}{n+1} P^{(1,-1)}_{n}(-3)
\quad\mbox{for $n\geq 2$}. 
$$
Using the swapping rule (\ref{E_jswap}) we obtain the following
remarkable variant of (\ref{E_dp1}).
\begin{equation}
\label{E_sp}
s_n=\frac{2}{n+1}\cdot P^{(-1,1)}_{n}(3)\quad\mbox{for $n\geq 1$}.
\end{equation} 
The same swapping rule for shifted Jacobi polynomials (\ref{E_sjswap})
allows to rewrite Corollary~\ref{C_wsn} as 
\begin{equation}
\label{E_wsn}
s_n^{u,v,w}=\frac{w^n
}{n+1}\left(1+\frac{w}{uv}\right)\sP^{(-1,1)}_{n}\left(\frac{uv}{w}+1\right)
\quad\mbox{for $n\geq 1$}.  
\end{equation} 
Central Delannoy numbers and Schr\"oder numbers satisfy an obvious recursion
formula, which may be generalized to their weighted variants as follows.
\begin{proposition}
\label{P_cdrec}
The weighted central Delannoy numbers and the weighted Schr\"oder
numbers satisfy the following formula.
$$
d_{n,n}^{u,v,w}=2uv\sum_{k=0}^{n-1} d_{k,k}^{u,v,w} s_{n-k-1}^{u,v,w}+w
d_{n-1,n-1}^{u,v,w}. 
$$
\end{proposition}
\begin{proof}
For any Delannoy path from $(0,0)$ to $(n,n)$ there is a unique $(k,k)$
on the path where the path contains a lattice point of the form $(i,i)$
for the last time before $(n,n)$. Consider the contribution of all
Delannoy paths associated to the same fixed $(k,k)$, where 
$k\in \{0,1,\ldots,n-1\}$. The part of all such paths up to $(k,k)$
contributes a factor of $d_{k,k}^{u,v,w}$. The rest must begin with an
east step and end with a north step, contributing a factor of $uv$. For
$k<n-1$ the remaining steps 
contribute a factor of $2 s_{n-k-1}^{u,v,w}$. The factor of $2$
represents the choice of staying strictly above or below the line $y=x$
for the rest of the path. For $k=n-1$, besides the contribution of $2uv
s_{0}^{u,v,w}=2uv$ 
the possibility of adding a northeast step from $(n-1,n-1)$ to $(n,n)$
arises: this possibility is accounted for by the term $w
d_{n-1,n-1}^{u,v,w}$.    
\end{proof}
Using Corollary~\ref{C_llp} and (\ref{E_sdef}) we obtain:
\begin{corollary}
The shifted Legendre polynomials satisfy the recursion formula
$$
\sP_n(x)=2x\sum_{k=0}^{n-1} \sP_{k}(x) S_{n-k-1}(x)-\sP_{n-1}(x). 
$$
\end{corollary}
Using Proposition~\ref{P_SP} the last corollary may be rewritten as
follows.
\begin{corollary}
We have 
$$
\sP_n(x)=2\sum_{k=0}^{n-2} \sP_{k}(x)
\frac{x-1}{n-k}\sP^{(1,-1)}_{n-k-1}(x)+(2x-1)\sP_{n-1}(x)\quad\mbox{and} 
$$
$$
P_n(x)=\sum_{k=0}^{n-2} P_{k}(x)
\frac{x-1}{n-k}P^{(1,-1)}_{n-k-1}(x)+xP_{n-1}(x)\quad\mbox{for
$n\geq 1$.} 
$$
\end{corollary}
Comparing (\ref{E_Sexp}) with (\ref{E_sP}) we find immediately that the
Schr\"oder polynomials are closely related to the the antiderivatives of
the shifted Legendre polynomials:
\begin{equation}
S_n(x)=\frac{1}{x}\int_0^x \sP_{n-1}(t)\ dt\quad\mbox{holds for $n\geq 1$.}
\end{equation}
By Proposition~\ref{P_SP} this implies
\begin{equation}
\frac{1}{n+1}(x-1)\sP^{(1,-1)}_{n}(x)=\int_0^x \sP_{n}(t)\
dt\quad\mbox{for $n\geq 1$.} 
\end{equation}
This equation may be generalized to the following statement.
\begin{proposition}
\label{P_int}
Let $n$ and $\alpha$ be positive integers. Applying the antiderivative
operator 
$$f(x)\mapsto \int_0^x f(t)\ dt$$ 
to $\sP_{n}(x)$ exactly $\alpha$ times
yields the polynomial
$\frac{1}{(n+1)_{\alpha}}(x-1)^{\alpha}\sP_n^{(\alpha,-\alpha)}(x)$.  
\end{proposition}
\begin{proof}
By Proposition~\ref{P_sj} we have 
$$
(x-1)^{\alpha}\sP^{(\alpha,-\alpha)}_n(x)=
\sum_{k=0}^{n+\alpha} (-1)^{n+\alpha-k} x^k \binom{n+\alpha}{k}
\binom{n-\alpha+k}{n}.$$
Here the binomial coefficient $\binom{n-\alpha+k}{n}$ is zero for
$k<\alpha$. Thus, after introducing $j:=k-\alpha$, we may write 
$$
(x-1)^{\alpha}\sP^{(\alpha,-\alpha)}_n(x)=
\sum_{j=0}^{n} (-1)^{n-j} x^{j+\alpha} \binom{n+\alpha}{j+\alpha}
\binom{n-j}{n}.$$
Thus
\begin{equation}
\label{E_a-a}
\frac{(x-1)^{\alpha}\sP^{(\alpha,-\alpha)}_n(x)}{(n+1)_{\alpha}}
=
\sum_{j=0}^{n} (-1)^{n-j} \frac{x^{j+\alpha}}{(j+1)_{\alpha}} \binom{n}{j}
\binom{n-j}{n}.
\end{equation}
Taking the derivative on both sides yields a right hand side of the same
form, with the value of $\alpha$ decreased by one.
Thus we obtain
$$
\frac{d}{dx}
\frac{(x-1)^{\alpha}\sP^{(\alpha,-\alpha)}_n(x)}{(n+1)_{\alpha}} 
=
\frac{(x-1)^{\alpha-1}\sP^{(\alpha-1,-(\alpha-1))}_n(x)}{
  (n+1)_{\alpha-1}} 
\quad\mbox{for $\alpha\geq 1$}. 
$$
The statement follows from applying this observation $\alpha$ times and
from the fact that $\sP^{(\alpha,-\alpha)}_n(0)=0$ holds for $n,\alpha>0$.  
\end{proof}
Proposition~\ref{P_int} has an interesting consequence for the {\em
  Narayana polynomials $N_n(x)$} defined by 
\begin{equation}
N_n(x) =\sum_{k=1}^n \frac{1}{n} \binom{n}{k-1}\binom{n}{k} x^k.
\end{equation}  
As it was shown by Mansour and Sun~\cite[Theorem 2.1]{Mansour-Sun}, 
the Narayana polynomials are related to the antiderivatives of the
shifted Legendre polynomials via the formula 
\begin{equation}
\label{eq:NsP}
N_n(x)=(x-1)^{n+1} \int_{0}^{\frac{x}{x-1}} P_n(2t-1) \ dt =(x-1)^{n+1} \int_{0}^{\frac{x}{x-1}} \sP_n(t) \ dt.
\end{equation}
Direct substitution of Proposition~\ref{P_int} into \eqref{eq:NsP} yields 
$$
N_n(x)=(x-1)^{n+1} \frac{1}{n+1}\left(\frac{x}{x-1}-1\right) 
\sP^{(1,-1)}\left(\frac{x}{x-1}\right).
$$
\begin{corollary}
The Narayana polynomials $N_n(x)$ are connected to the shifted Jacobi
polynomials $\sP^{(1,-1)}(x)$ by the formula
$$
N_n(x)=\frac{(x-1)^n}{n+1} \sP^{(1,-1)}\left(\frac{x}{x-1}\right). 
$$
\end{corollary}
I am indebted to Yidong Sun for helping me realize this connection. 

We conclude this section by observing that the polynomials
$\{S_n(x)\}_{n\geq 0}$ {\em almost} form an orthogonal
polynomial sequence. In fact, it is possible to show that the monic
polynomials  
$$p_{n}(x):=\frac{1}{\binom{2n}{n}} \frac{x-1}{x}\sP^{(1,-1)}_n(x)  
$$
satisfy Favard's recursion formula~(\ref{E_Favrec})
$$
p_{n}(x)
=\left(x-\frac{1}{2}\right)p_{n-1}(x)
-\frac{n(n-2)}{4(2n-1)(2n-3)}p_{n-2}(x)  
\quad\mbox{for $n\geq 2$}. 
$$
Unfortunately, substituting $n=2$ yields $\lambda_2=0$ thus the moment
functional will not be quasi-definite.

\section{Concluding remarks}
\label{sec:concl}

One of the main observations in this paper is that the connection
(\ref{E_dp}) between the central Delannoy numbers and Legendre polynomials
may be extended to all Delannoy numbers and the Jacobi polynomials
$\{P_n^{(\alpha,0)}(x)\}_{n\geq 0}$ via (\ref{E_dp1}). An eerily similar
relation was found in~\cite[(3.2)]{Hetyei-dn} between the {\em modified
  Delannoy numbers $\adn_{m,n}$} and the Jacobi polynomials
$\{P_n^{(0,\beta)}(x)\}_{n\geq 0}$. Formula (3.2) in~\cite{Hetyei-dn}  
may be restated as 
\begin{equation}
\label{E_dp2}
\adn_{n+\beta,n}=P_{n}^{(0,\beta)}(3) \quad\mbox{for $\beta\in{\mathbb N}$}.
\end{equation}
The modified Delannoy number $\adn_{m,n}$ is the number of lattice paths
  from $(0,0)$ to $(m,n+1)$ whose steps belong to ${\mathbb N}\times
  {\mathbb P}$. (Here ${\mathbb P}$ denotes the set of positive integers.)
It would be worth exploring whether the weight enumeration of modified
  Delannoy paths (defined as having steps from ${\mathbb N}\times
  {\mathbb P}$) could also be done using shifted Jacobi polynomials, and
  whether there is a {\em duality} between the two theories that extends
  the (``signless'') swapping of the parameters observed above in
  (\ref{E_dp1}) and (\ref{E_dp2}).

The fact that the (shifted) Legendre and Jacobi polynomials are orthogonal with
respect to {\em some} weight function is also a consequence of
Favard's theorem. The question naturally arises whether 
Viennot's combinatorial theory~\cite{Viennot} could be applied to find
a moment functional with respect to
which they are orthogonal. Laguerre polynomials (their usual monic
variant) have a very nice combinatorial interpretation in Viennot's
work. For (the monic version of the) Legendre and Jacobi polynomials,
however, the coefficients 
$\{c_n\}_{n\geq 1}$ and $\{\lambda_n\}_{n\geq 1}$ are not all
integers. For example, the monic variant of the Legendre polynomials is
given by 
$$
p_n(x):=\frac{2^n P_n(x)}{\binom{2n}{n}}.
$$
Favard's recursion formula~(\ref{E_Favrec}) takes the form 
$$
p_n(x)=x p_{n-1}x-\frac{(n-1)^2}{(2n-1)(2n-3)} p_{n-2}(x).
$$
Thus some weights $-c_n$ and $-\lambda_n$ used on Viennot's {\em Favard paths}
have to be fractions, and the formula expressing the moments as total
weights of Motzkin paths seems very hard to evaluate. 
For Legendre polynomials, the horizontal steps will have zero weight,
the northeast steps $(1,1)$ will have weight $1$, the southeast steps
$(1,-1)$ will have weight $k^2/(4k^2-1)$ if they start at a point whose
second coordinate $k$.  Using the fact that Legendre polynomials are orthogonal
with respect to the inner product 
$$\langle f, g\rangle :=\int_{-1}^1 f(x)\cdot g(x)\ dx,$$ 
we know that the moments must be a constant multiple of
$$\int_{-1}^1 \frac{1}{x^n}\ dx=
\left\{
\begin{array}{rr}
0\quad\mbox{if $x$ is odd,}\\
\frac{2}{n+1}\quad\mbox{if $x$ is even.}
\end{array}\\
\right.$$ 
Experimental evaluation of Viennot's Motzkin paths for the first few
even values of $n$ shows that the total weight of these paths is
$1/(n+1)$. A direct combinatorial proof of this fact would be
desirable. It should also be noted that the non-monic variant
$$
q_n(x):=\frac{2^n (2n-1)!! P_n(x)}{\binom{2n}{n}}
$$
of the Legendre polynomials satisfies the recursion formula
\begin{equation}
q_n(x)=(2n-1)x q_{n-1}(x)-(n-1)^2 q_{n-2}(x)
\end{equation}
with integer connecting coefficients. Perhaps a non-monic version of
Favard's theorem is worth considering, to which Viennot's combinatorial
proof could be extended by replacing the weight $x$ on the short
vertical steps in his Favard paths by a weight of $(2n-1)x$ where $n$
is determined by the height of the step. 

We should also mention the combinatorial model of the Jacobi polynomials
constructed by Foata and Leroux~\cite{Foata-Leroux}, whose study was
continued by Leroux and Strehl~\cite{Leroux-Strehl}, \cite{Strehl1},
\cite{Strehl3}. This model was never used to prove the orthogonality of
the Jacobi polynomials, and it seems hard to be used for such a
purpose. On the other hand, finding an interpretation for the Delannoy
numbers in it seems to be an interesting question for future research. 

As seen in Section~\ref{s_rook}, our current work relates the Legendre
and Jacobi polynomials to the classical theory of rook polynomials. The
level of this connection is not satisfying, since we only found a
binomial identity to which the orthogonality of the Laguerre polynomials
(rook theory variant) and the orthogonality of certain Jacobi
polynomials may both be reduced. It would be desirable to find results 
directly relating colored Delannoy paths and rook placements as
combinatorial structures.

Our proof of the orthogonality of the Jacobi
polynomials $P_n^{(\alpha,\beta)}(x)$  for nonzero $\alpha$ depends on
the algebraic formula given in Proposition~\ref{P_abdec}. A more direct
combinatorial extension to the case $\alpha>0$ would be desirable, which
can perhaps be found once the ``duality'' between $P_n^{(\alpha,0)}(x)$
and $P_n^{(0,\beta)}(x)$ is better understood. Although
Proposition~\ref{P_int} seems to give a ``philosophical reason'' to
focus on the polynomials 
$(x-1)^{\alpha}\sP_n^{(\alpha,-\alpha)}(x)/(n+1)_{\alpha}$, these
have a combinatorial interpretation for $\alpha=1$ only, for 
higher values we start having non-integer coefficients. It is remarkable
though that the same polynomials $(x-1)^{\alpha}\sP_n^{(\alpha,-\alpha)}(x)$,
without being divided by  $(n+1)_{\alpha}$ play a key role in our
work and have a combinatorial interpretation for all $\alpha\in
{\mathbb N}$.   

Finally we wish to note that there is a yet to be explored connection
with the theory of {\em Riordan arrays} (see Sprugnoli~\cite{Sprugnoli} for
definition and uses). The weighted Delannoy number $d_{m,n}^{u,v,w}$ is the
coefficient of $t^n$ in $(u+wt)^m/(1-vt)^{m+1}$. An immediate
consequence of this observation is that the $n$-th row $k$-th column
entry in the Riordan array $(1/(1-vt),t(u+wt)/(1-vt))$ is
$d_{k,n-k}^{u,v,w}$. The numbers $d_{m,n}^{1,2,-1}$ appear as entry A1016195 
in Sloane~\cite{Sloane}, listing the entries of the Riordan array
$(1/(1-2t),t(1-t)/(1-2t))$. Our results should allow to write summation
formulas for Jacobi polynomials using the theory of Riordan arrays. 

\section*{Acknowledgments}
I wish to thank Mireille Bousquet-M\'elou for acquainting me with
Viennot's work~\cite{Viennot}. Many thanks to Yongdo Lim for bringing
the work of Fray and Roselle~\cite{Fray-Roselle} to my attention and for
an encouraging and stimulating correspondence. I am indebted to Tom
Koornwinder for helping me understand Romanovski polynomials and to
Yidong Sun for bringing the connection with Narayana polynomials to my
attention. This work was supported
by the NSA grant \# H98230-07-1-0073, and is dedicated to the memory of
Pierre Leroux, who himself did some important
work~\cite{Foata-Leroux,Leroux-Strehl} on the combinatorics of Jacobi
polynomials.

\end{document}